\documentclass[12pt]{amsart}

\usepackage{pdfpages}

\usepackage{amsmath,amsthm,amsfonts,amssymb,times,latexsym,enumerate,comment,marginnote,slashed,dsfont}
\usepackage{pdfpages, relsize, setspace}
\usepackage{xypic}
\usepackage[OT2,T1]{fontenc} 

\DeclareSymbolFont{cyrletters}{OT2}{wncyr}{m}{n}
\DeclareMathSymbol{\sha}{\mathalpha}{cyrletters}{"58}

\newcommand{\overtext}{\overset}

\newcommand{\mbf}{\mathbb{F}}
\newcommand{\mbr}{\mathbb{R}}
\newcommand{\mbz}{\mathbb{Z}}

\newcommand{\mbc}{\mathbb{C}}

\newcommand{\mbp}{\mathbb{P}}

\renewcommand{\:}{\colon}

\newcommand{\ra}{\rightarrow}

\newcommand{\iso}{\cong}

\newcommand{\bs}{\backslash}
\newcommand{\ssq}{\subseteq}

\DeclareMathOperator{\Pic}{Pic}
\DeclareMathOperator{\Div}{Div}

\newcommand{\wtilde}{\widetilde}

\DeclareMathOperator{\Aut}{Aut}

\renewcommand{\div}{\operatorname{div}}

\newcommand*{\longhookrightarrow}{\ensuremath{\lhook\joinrel\relbar\joinrel\rightarrow}}
\newcommand*{\inj}{\longhookrightarrow}

\newcommand{\bbm}{\begin{bmatrix}}
\newcommand{\ebm}{\end{bmatrix}}
\newcommand{\bpm}{\begin{pmatrix}}
\newcommand{\epm}{\end{pmatrix}}
\newcommand{\xym}[1]{\xymatrix{#1}}

\newcommand{\gen}[1]{\langle #1 \rangle}

\DeclareMathOperator{\im}{Im}
\renewcommand{\Im}{\im}

\newcommand{\set}[1]{\left \{#1 \right \}}

\newcommand{\Sym}{\text{Sym}}

\newtheorem{theorem}{Theorem}[section]
\newtheorem{lemma}[theorem]{Lemma}
\newtheorem{proposition}[theorem]{Proposition}
\newtheorem*{theorem*}{Theorem}
\newtheorem{corollary}[theorem]{Corollary}

\newtheorem{example}[theorem]{Example}

\makeatletter
\newtheorem*{rep@theorem}{\rep@title}
\newcommand{\newreptheorem}[2]{%
	\newenvironment{rep#1}[1]{%
		\def\rep@title{#2 \ref{##1}}%
		\begin{rep@theorem}}%
		{\end{rep@theorem}}}
\makeatother

\newreptheorem{theorem}{Theorem}
\newreptheorem{conjecture}{Conjecture}

\theoremstyle{definition}
\newtheorem{definition}[theorem]{Definition}

\newtheorem{remark}[theorem]{Remark}

\usepackage{listings}
\usepackage[alphabetic]{amsrefs}

\usepackage{stmaryrd}

\DeclareMathOperator{\rank}{rk}

\DeclareMathOperator{\GL}{GL}
\DeclareMathOperator{\SL}{SL}
\DeclareMathOperator{\SO}{SO}

\DeclareMathOperator{\Hom}{Hom}
\DeclareMathOperator{\PGL}{PGL}
\DeclareMathOperator{\Lie}{Lie}
\DeclareMathOperator{\trace}{Tr}

\DeclareMathOperator{\ad}{ad}
\DeclareMathOperator{\Stab}{Stab}

\newcommand{\mbg}{\mathbb{G}}
\newcommand{\mcDP}{\mathcal{DP}}
\newcommand{\mcS}{\mathcal{S}}

\newcommand{\mcT}{\mathcal{T}}

\newcommand{\Gm}{\mathbb{G}_m}
\newcommand{\Ga}{\mathbb{G}_a}
\newcommand{\sep}{\mathrm{sep}}
\newcommand{\reg}{\mathrm{reg}}
\newcommand{\rss}{\mathrm{rss}}
\newcommand{\ram}{\mathrm{ram}}
\newcommand{\sram}{\mathrm{s.ram}}
\newcommand{\tram}{\mathrm{t.ram}}
\newcommand{\node}{\mathrm{node}}
\newcommand{\cusp}{\mathrm{cusp}}
\newcommand{\id}{\mathrm{id}}

\newcommand{\heisenberg}[1]{\wtilde H_{#1} }
\newcommand{\catname}[1]{{\normalfont\mathbf{#1}}}

\newcommand{\exactseq}[4][]{\xym{ 0 \ar[r] & #2 \ar[r] & #3 \ar[r]^{#1} & #4 \ar[r] & 0}}

\newcommand{\arhook}{\ar@{^{(}->}}

\mathchardef\mhyphen="2D

\newcommand{\mapdef}[5]{
	\begin{eqnarray*} 
		#1 \: & #2 & \ra #3 \\  
		\	& #4 & \mapsto #5 
	\end{eqnarray*}
}

\newcommand{\leftmarginhack}{ccccccccc}

\xyoption{rotate}

\title{An arithmetic invariant theory of curves from $E_8$}
\author{Avinash Kulkarni}

\thanks{The author was partially supported by NSERC}

\date{\today}

\begin{document}

\begin{abstract}
	Let $k$ be a field of characteristic $0$, let $C/k$ be a uniquely trigonal genus $4$ curve, and let $P \in C(k)$ be a simply ramified point of the uniquely trigonal morphism. We construct an assignment of an orbit of an algebraic group of type $E_8$ acting on a specific variety to each element of $J_C(k)/2$. The algebraic group and variety are independent of the choice of $(C,P)$. We also construct a similar identification for uniquely trigonal genus $4$ curves $C$ with $P \in C(k)$ a totally ramified point of the trigonal morphism.
	
	Our assignments are analogous to the assignment of a genus $3$ curve with a rational point $(C,P)$ to an orbit of an algebraic group of type $E_7$ exhibited by Jack Thorne \cite{thorne2016arithmetic}. Our assignment is also analogous to one constructed by Bhargava and Gross \cite{bhargava2013average}, who use it determine average ranks of hyperelliptic Jacobians.
\end{abstract}

\maketitle


\section{Introduction}
	
	Given a smooth projective curve over a field $k$, we are often interested in computing the $k$-rational points on this curve or its Jacobian variety. If $k$ is a number field, the $k$-rational points on the Jacobian variety has the structure of a finitely generate abelian group, and performing a $2$-descent allows us to compute an upper bound for the rank. More specifically, performing a $2$-descent computes the rank of the $2$-Selmer group of the Jacobian, which is necessarily at least the rank of the Jacobian. If the curve $C$ has a $k$-rational point, then the Abel-Jacobi map $j$ embeds $C$ into its Jacobian $J_C$. A two-cover descent \cite{bruin2009twocover} on $C$ identifies a subset of the $2$-Selmer group of $J_C$ which contains those classes in the image of $C(k) \overtext{j}{\ra} J_C(k) \overtext{\delta}{\ra} H^1(k, J_C[2])$.
	
	For the family of hyperelliptic curves over $k$ with a marked rational Weierstrass point, whose models are given by
		\[
			C\: y^2 = x^{2n+1} + a_{2n}x^{2n} + \ldots + a_1 x + a_0
		\] 
	 it was shown by Bhargava and Gross in \cite{bhargava2013average} that we can understand the average size of the $2$-Selmer group of the Jacobian in a meaningful sense. One of the major insights was that the elements of the $2$-Selmer group naturally correspond to orbits of $\SO_{2n+1}(k)$ acting by conjugation on the linear space $\{A \in \mathfrak{sl}_{2n+1}k : A = -A^T \}$ \cite{bhargava2013average}, \cite[Section 2.2, Example]{thorne2013vinberg}. Using a representation of $\SL_n k$ on $k^2 \otimes \Sym^n k$, one can understand the average behaviour of $k$-rational points on general hyperelliptic curves \cite{bhargava2017positive}. 
	
	The goal of arithmetic invariant theory is to describe the arithmetic of a family of algebraic or geometric objects in terms of the space of $k$-rational orbits of a representation of an algebraic group. In recent years it has been extremely successful in understanding the average behaviour of $k$-rational points on certain curves and abelian varieties \cite{bhargava2013average, bhargava2017positive, bhargava2015elliptic, bhargava2017average5selmer, rains2017invariant, romano2017singularities, wang2013pencils}. The results for hyperelliptic curves extend to a connection between the adjoint  representations of groups of type $A,D,E$ and specifically identified families of curves \cite{thorne2013vinberg, romano2017singularities} which are constructed from the invariants of these representations. Thorne has also demonstrated that the adjoint form of $E_6$ or $E_7$ acting on itself via conjugation gives rise to an understanding of $2$-descents on plane quartic curves \cite{thorne2016arithmetic}.
	
	The purpose of this paper is to establish a connection between the split adjoint simple group of type $E_8$ acting on itself via conjugation and the family of uniquely trigonal genus $4$ curves with a ramified fibre defined over $k$. Our method of argument closely follows \cite{thorne2016arithmetic}. A large part of our work is adapting many of the results for plane quartic curves that \cite{thorne2016arithmetic} relies upon to apply to uniquely trigonal genus $4$ curves. Additionally, there are some minor technical differences between $E_7$ and $E_8$ that we need to take into account.

	\begin{remark}
		The family of affine plane curves in \cite{romano2017singularities} given by
			\begin{equation} \label{eqn:tho15fam}
				y^3	= x^5 +	y(c_2 x^3 + c_8 x^2 + c_{14} x + c_{20}) + c_{12} x^3 + c_{18} x^2 + c_{24} x + c_{30} \tag{$\star$}
			\end{equation}
		is in fact a family of uniquely trigonal genus $4$ curves. We point out that our construction treats a more general curve as  the trigonal morphism $\pi\: (y,x) \ra (x:1)$ will always have a totally ramified fibre over the point at infinity for any curve of the form (\ref{eqn:tho15fam}). Our procedure naturally identifies how the family in \cite{romano2017singularities} connects to the general case. Specifically, the family of pairs $(C,P)$ with $C$ a curve in (\ref{eqn:tho15fam}) and $P$ the totally ramified point on $C$ at infinity is treated in the $\mathfrak{e}_8$ case of our argument. 	Moreover, if $C/k$ is a uniquely trigonal curve of the form (\ref{eqn:tho15fam}), our results augment \cite[Theorem 2.10]{romano2017singularities} by constructing an orbit for every class in $J_C(k)/2$, as opposed to just those classes in the image of the Abel-Jacobi map, i.e, those classes relevant to a two-cover descent. Unlike \cite{romano2017singularities}, we do not attempt to calculate the average size of the Selmer group of the Jacobian of a uniquely trigonal genus $4$ curve with a marked ramified fibre. We hope that this is the topic of a future paper.
	\end{remark}

	We describe the layout of the paper. We fix notation in Subsection \ref{subsec: notation}. Our main results are stated in Subsection \ref{subsec: main statements}. In Section \ref{sec: Thorne review} we give a brief summary of the techniques introduced in \cite{thorne2016arithmetic}. We state the background results we will need from Lie groups in Section \ref{sec: Background A,D,E} and describe the family of uniquely trigonal genus four curves in Section \ref{sec: uniquely trigonal locus}. In Section \ref{sec: uniquely trigonal locus} we also prove Theorem \ref{conj: e8 3.4}. In Section \ref{sec: Reeder calculation} we perform some technical calculations. Finally, we prove Theorem \ref{conj: e8 3.5} and Theorem \ref{conj: e8 3.6} in Section \ref{sec: Construction of orbits}.

	\subsection{Notation and conventions} \label{subsec: notation}

	Throughout we shall assume that $k$ is a field of characteristic $0$. The separable closure of $k$ will be denoted by $k^\sep$. If $V$ is a variety defined over $k$ we let $V_{k^\sep} := V \times_k k^\sep$. We implicitly mean $k$-rational whenever the term rational is used. We use the standard terminology from the theory of Lie groups and Lie algebras, which can be found in \cite{fulton1991representation} or \cite[Chapter 7, Section 5]{bourbaki1975lie}. Generally, $H$ will denote a split adjoint simple group of type $A$, $D$, or $E$ over $k$ and $T$ will denote a split maximal subtorus of $H$. Additionally, $\Lambda$ will denote the character lattice of $T$ and $W_T$ will denote the Weyl group associated to $\Lambda$, acting on $T$. If $H$ is a Lie group, then we will denote by $\mathfrak{h}$ its Lie algebra. 
	
	We let $\catname{Set}$ be the category of sets and let ${\catname{ k\mhyphen alg}}$ be the category of $k$-algebras. In this article a \emph{moduli functor} is a functor from ${\catname{ k\mhyphen alg}}$ to ${\catname{Set}}$. We denote the moduli functor for smooth uniquely trigonal genus $4$ curves by $\mathcal{T}_4$. We denote the moduli functor for smooth uniquely trigonal genus 4 curves, together with a marked ramification point, simple ramification point, or totally ramified point by $\mcT_4^\ram$, $\mcT_4^\sram$, or $\mcT_4^\tram$ respectively. We denote the moduli functor for smooth genus $3$ curves by $\mathcal{M}_3$, and the moduli functor for pointed smooth genus $3$ curves by $\mathcal{M}_3^1$. i.e, $\mathcal{M}_3^1$ is the moduli functor of smooth genus $3$ curves together with a marked point. If $\mathcal{F}$ is one of the aforementioned moduli functors of curves then we denote by $\mathcal{F}(2)$ the moduli functor of curves parametrized by $\mathcal{F}$ with marked $2$-torsion of their Jacobian variety.
	
	A semisimple Lie algebra $\mathfrak{h}$ defined over $k$ is \emph{split} over $k$ if it has a splitting Cartan subalgebra $\mathfrak{t}$ defined over $k$. A Cartan subalgebra $\mathfrak{t} \inj \mathfrak{h}$ is \emph{splitting} (over $k$) if $\ad_\mathfrak{h} t$ is diagonalizable over $k$ for all $t \in \mathfrak{t}$. A Lie group is \emph{split} over $k$ if its Lie algebra is split over $k$. For a semisimple Lie group $H$, an element $x \in H$ is \emph{regular} if its centralizer in $H$ has minimal dimension. Similarly, for a semisimple Lie algebra $\mathfrak{h}$, an element $x \in \mathfrak{h}$ is \emph{regular} if its centralizer in $\mathfrak{h}$ has minimal dimension. If $G$ is a semisimple Lie group we denote the open subset of regular semisimple elements by $G^\rss$. We make a similar definition for $\mathfrak{g}^\rss$. The regular semi-simple points on a maximal subtorus $T$ of $H$ are exactly those not lying on a root hyperplane.

	If $\theta\: H \ra H$ is an involution of an algebraic group $H$, we define $H^\theta$ to be the subvariety of $H$ fixed by $\theta$ and define $H^{\theta(h) = h^{-1}}$ to be the subvariety of $H$ defined by the equation $\theta(h) = h^{-1}$. We write $H^\circ$ for the connected component of the identity of $H$. If $\mathfrak{h}$ is a split semisimple Lie algebra over $k$ of type $E_n$, then a split adjoint simple group of type $E_n$ over $k$ is any Lie group isomorphic over $k$ to $\Aut(\mathfrak{h})^\circ$ (c.f \cite[page 101 and Proposition D.40]{fulton1991representation}).
  
  	If $X/k$ is a scheme acted upon by a reductive group $G/k$, then we denote the categorical quotient of $X$ by $G$, with $X$ interpreted as functor from $\catname{ k\mhyphen alg}$ to ${\catname{Set}}$, by $X \sslash G$ \cite{mumford1994git}.
  
	\subsection{Statement of our main results} \label{subsec: main statements}
	
	Let $C/k$ be a uniquely trigonal genus $4$ curve and let $P \in C(k)$ be a ramification point of the unique (up to $\PGL_2$) morphism $\pi\: C \ra \mbp^1$ of degree $3$. We split our analysis into two cases:
	
	\begin{enumerate}[\leftmarginhack]
		\item[ Case $E_8$:] The ramification index of $P$ with respect to $\pi$ is $2$ (the generic case, where $P$ is a simple ramification point).
		
		\item[ Case $\mathfrak{e}_8$:] The ramification index of $P$ with respect to $\pi$ is $3$ ($P$ is a totally ramified point of $\pi$).
	\end{enumerate}
	
	As in ~\cite{thorne2016arithmetic}, the names indicate the Lie group or Lie algebra used to construct the appropriate orbit spaces. We prove the direct analogues of Thorne's results.

	\begin{theorem} \label{conj: e8 3.4}
		\
		\begin{enumerate}
			\item[$E_8\:$]
			If $k=k^\sep$, then there is a bijection
			\[
			\mcT^\sram(k) \ra (T^\rss \sslash W_T)(k)
			\]
			with $\Lambda$ the root lattice of type $E_8$ and $T = \Hom(\Lambda, \Gm) $ the split torus of rank $8$.
			
			\item[$\mathfrak{e}_8\:$]
			If $k=k^\sep$, then there is a bijection
			\[
			\mcT^\tram(k) \ra (\mbp \mathfrak{t}^\rss \sslash W_T)(k)
			\]
			with $\mathfrak{t}$ the Lie algebra of the torus $T$.
		\end{enumerate}
	\end{theorem}

	\begin{theorem} \label{conj: e8 3.5}
		\
		\begin{enumerate}
			\item[$E_8\:$] 
			Let $H$ be the split adjoint simple group of type $E_8$ over $k$, let $\theta$ be an involution of $H$ satisfying the conditions of Proposition \ref{prop: thorne involution}, and let $X := (H^{\theta(h)=h^{-1}})^\circ$ be the theta inverted subvariety. Let $G := (H^\theta)^\circ$. Let $X^\rss$ be the open subset of regular semi-simple elements. Then the assignment $(C,P) \mapsto \kappa_C$ of Theorem \ref{conj: e8 3.4}-$E_8$ determines a map
			\[
			\mcT^\sram(k) \ra G(k) \bs X^\rss(k).
			\]
			If $k=k^\sep$, then this map is a bijection.	
			\item[$\mathfrak{e}_8\:$] 
			Let $H$ be as above, let $X := \mathfrak{h}^{d\theta=-1}$ be the Lie algebra of $(H^{\theta(h)=h^{-1}})^\circ$, and let $X^\rss$ be the open subset of regular semi-simple elements. Then the assignment $(C,P) \mapsto \kappa_C$ of Theorem \ref{conj: e8 3.4}-$\mathfrak{e}_8$ determines a map
			\[
			\mcT^\tram(k) \ra G(k) \bs \mbp X^\rss(k).
			\]
			If $k=k^\sep$, then this map is a bijection.
		\end{enumerate}
	\end{theorem}

	\begin{theorem} \label{conj: e8 3.6}
		\
		\begin{enumerate}
			\item[$E_8\:$] 
			Fix an $x = (C,P) \in \mcT^\sram(k)$ and by abuse of notation we denote by $x$ the image of $(C,P)$ in $T \sslash W_T$. Let $\pi\: X \ra X \sslash G$ denote the natural quotient map, where $X$ is as in Theorem \ref{conj: e8 3.5}-$E_8$. Note that $X \sslash G$ is canonically isomorphic to $T \sslash W_T$. Let $X_x$ be the fibre of $\pi$ over $x$ and let $J_x = J_C$. Then there is a canonical injective map
			\[
			\frac{J_x(k)}{2} \inj G(k) \bs X_x(k).
			\]
			\item[$\mathfrak{e}_8\:$] 
			Fix an $x = (C,P) \in \mcT^\tram(k)$ and by abuse of notation we denote by $x$ the image of $(C,P)$ in $\mbp \mathfrak{t} \sslash W_T$. Let $\pi\: \mbp X \ra \mbp X \sslash G$ denote the natural quotient map, where $X$ is as in Theorem \ref{conj: e8 3.5}-$\mathfrak{e}_8$. Let $\mbp X_x$ be the fibre of $\pi$ over $x$ and let $J_x = J_C$. Then there is a canonical injective map
			\[
			\frac{J_x(k)}{2} \inj G(k) \bs \mbp X_x(k).
			\]
		\end{enumerate}
	\end{theorem}

\section{A summary of Thorne's construction of orbits} \label{sec: Thorne review}

	In this section, we give a summary of Thorne's construction of orbits since we will ultimately be following the same overarching strategy. The identification of pointed plane quartic curves to orbits is split into four cases depending on the behaviour of the tangent line at $P$ to the curve. These are:
	
	\begin{enumerate}[\leftmarginhack]
		\item[Case $E_7$:] The tangent line to $P$ meets $C$ at exactly 3 points (the generic case).
		
		\item[Case $\mathfrak{e}_7$:] The tangent line to $P$ meets $C$ at exactly 2 points, with contact order 3 at $P$ (i.e $P$ is a flex).

		\item[Case $E_6$:] The tangent line to $P$ meets $C$ at exactly 2 points, with contact order 2 at $P$ (i.e $P$ lies on a bitangent).

		\item[Case $\mathfrak{e}_6$:] The tangent line to $P$ meets $C$ at exactly 1 point, with contact order 4 (i.e $P$ is a hyper-flex)			
	\end{enumerate} 
	
	The name of each case refers to the semisimple algebraic Lie group, or semisimple algebraic Lie algebra used to construct the orbit space which parametrizes points of the noted type on plane quartic curves. 

	\bigskip

	For simplicity we describe the argument of ~\cite{thorne2016arithmetic} for the $E_7$ case only, but we note that the other cases are treated similarly. The argument  is guided along three milestone theorems. We list these as Theorem \ref{thm: thorne3.4}, Theorem \ref{thm: thorne3.5}, and Theorem \ref{thm: thorne3.6} and outline the arguments used in proving them. Let $\mathcal{S}$ be the subfunctor of $\mathcal{M}_3^1$ parameterizing non-hyperelliptic pointed genus $3$ curves, whose marked point $P$ in the canonical model is not a flex and does not lie on a bitangent.
	
	\begin{theorem}[{\cite[Theorem 3.4]{thorne2016arithmetic}}] \label{thm: thorne3.4}
		If $k=k^\sep$, then there is a bijection
		\[
			\mcS(k) \ra (T^\rss \sslash W_T)(k)
		\]
		with $\Lambda$ the root lattice of type $E_7$ and $T = \Hom(\Lambda, \Gm) $ the split torus of rank $7$.
	\end{theorem}

	Theorem \ref{thm: thorne3.4} is a reformulation of some results of ~\cite{looijenga1993cohomology}. The result is established by using a connection between non-hyperelliptic genus $4$ curves and del Pezzo surfaces of degree $2$. Theorem \ref{thm: thorne3.4} or \cite[Proposition 1.8]{looijenga1993cohomology} shows that $\mathcal{S}(2)$ is naturally isomorphic to an open subset of a rank $7 \ (= \dim \mathcal{M}_3^1)$ torus. This isomorphism arises from the relationship between smooth plane quartics and del Pezzo surfaces of degree $2$. Namely, every del Pezzo surface of degree $2$ is a double cover of $\mbp^2$ branched along a smooth plane quartic, and every smooth plane quartic arises in this way. The result in \cite[Proposition 1.8]{looijenga1993cohomology} shows that the additional data of an anti-canonical section of a degree $2$ del Pezzo surface corresponds (generically) to the data of a point on the associated plane quartic.
	 
	Thorne then takes advantage of the fact that del Pezzo surfaces of degree $2$ have a strong connection to split adjoint simple groups of type $E_7$ and uses classical results on Lie groups to strengthen Theorem \ref{thm: thorne3.4} to Theorem \ref{thm: thorne3.5} below. In particular, ~\cite[Theorem 1.11]{thorne2016arithmetic} gives an isomorphism $T^\rss \sslash W_T \iso X^\rss \sslash G$, and ~\cite[Section 2]{thorne2016arithmetic} provides a construction of the larger algebraic group using a root datum and auxiliary data. The final step of this argument is to show that this extra abstract data is supplied by the choice of plane quartic and point.

	\begin{theorem}[{\cite[Theorem 3.5]{thorne2016arithmetic}}] \label{thm: thorne3.5}
		Let $H$ be the split adjoint simple group of type $E_7$ over $k$, let $\theta$ be an involution of $H$ satisfying the conditions of Proposition \ref{prop: thorne involution}, and let $X := (H^{\theta(h)=h^{-1}})^\circ$ be the theta inverted subvariety. Let $G := (H^\theta)^\circ$. Let $X^\rss$ be the open subset of regular semi-simple elements. Then the assignment $(C,P) \mapsto \kappa_C$ of Theorem \ref{thm: thorne3.4} determines a map
		\[
			\mcS(k) \ra G(k) \bs X^\rss(k).
		\] 
	\end{theorem}
		
	Finally, Theorem \ref{thm: thorne3.5} is twisted to become Theorem \ref{thm: thorne3.6}. The auxiliary data provided by a plane quartic and point $(C,P)$ in the proof of Theorem \ref{thm: thorne3.5} depends on a chosen translate of $2W_{g-1}$, where $W_{g-1}$ is a theta divisor of $J_C$. In turn, for a fixed curve $C$, the choices of non-equivalent translates of $2W_{g-1}$ over a field $k$ are parametrized by the choices of $[D] \in \frac{J_C(k)}{2}$. In Theorem \ref{thm: thorne3.5} this choice can be made systematically among all pointed plane quartics since $\frac{J_C(k)}{2}$ always contains the trivial class\footnote{As one is parameterizing plane quartic curves with a $k$-rational point, one could also systematically choose the class $[4P - \mathcal{K}_C]$ among all pointed curves, where here $\mathcal{K}_C$ is the canonical class of $C$.}. However, for a fixed pointed curve $(C,P)$ there are potentially many choices of a class in $\frac{J_C(k)}{2}$, and thus many choices of translates of $2W_{g-1}$. The $G(k^\sep)$-orbit assigned to $(C,P)$ decomposes into several $G(k)$-orbits, and there is an assignment of one of these $G(k)$-orbits to each choice of translate of $2W_{g-1}$. For the technical details one should consult \cite[Section 1]{thorne2016arithmetic} and the proof of Theorem \ref{thm: thorne3.6}.

	\begin{theorem}[{\cite[Theorem 3.6]{thorne2016arithmetic}}] \label{thm: thorne3.6}
		Fix an $x = (C,P) \in \mcS(k)$ and by abuse of notation we denote by $x$ the image of $(C,P)$ in $T \sslash W_T$. Let $\pi\: X \ra X \sslash G$ denote the natural quotient map. Note that $X \sslash G$ is canonically isomorphic to $T \sslash W_T$. Let $X_x$ be the fibre of $\pi$ over $x$. Then there is a canonical injection
			\[
				\frac{J_x(k)}{2} \inj G(k) \bs X_x(k)
			\]
		such that the image of $x$ in $G(k) \bs X(k)$ is the image of $[0] \in J_x(k)/2J_x(k)$. 
	\end{theorem}
	
	\begin{figure}  
		\[
			\xym{
				& (C,P) =: x \ar[r] & \mcS(k) \arhook[d] \arhook[r] & (T \sslash W_T)(k) \\
				\frac{J_x(k)}{2} \arhook[r] &  G(k) \bs X_x(k) \arhook[r] & G(k) \bs X(k)  \ar[r]^\rho & (X \sslash G)(k) \ar[u]^[@]{\sim} \\
				& &  & X \ar[u]^\pi
			}
		\] 
		\caption{Diagram of functors.} \label{fig: functors}
	\end{figure}

	Figure \ref{fig: functors}  is a diagrammatic description of the relations between the functors noted in the three theorems above. Note in general that $\rho$ is not a bijection, though it is a bijection when $k$ is separably closed.


\section{Background regarding groups of type $A$,$D$,$E$} \label{sec: Background A,D,E}

	We will make use of the results of \cite[Section 1D]{thorne2016arithmetic}, which we state here for convenience. Our Theorem 3.2 below is slightly modified from Thorne's Theorem 1.10. The change allows us to state our results in the $\mathfrak{e}_8$ case without the extra data of a tangent vector required at various points for the $\mathfrak{e}_6$, $\mathfrak{e}_7$ cases of Thorne's results.

	\begin{theorem}[{\cite[Theorem 1.11]{thorne2016arithmetic}}] \label{thm: thorne1.11}
		\begin{enumerate}[(a)]
			\item
			Let $H$ be a split adjoint simple group of type $A$,$D$, or $E$. Let $Y := (H^{\theta(h) = h^{-1}})^\circ$ and let $G := (H^\theta)^\circ$. Let $T$ be a maximal subtorus of $Y$ and let $W_T$ be the Weyl group of this torus. Then the inclusion $T \ssq Y$ induces an isomorphism
			\[
				T \sslash W_T \iso Y \sslash G.
			\]
			\item 
			Suppose that $k=k^\sep$ and let $x,y$ be regular semisimple elements. Then $x$ is $G(k)$-conjugate to $y$ if and only if $x,y$ have the same image in $Y \sslash G$.
			\item 
			There exists a discriminant polynomial $\Delta \in k[Y]$ such that for all $x \in Y$, $x$ is regular semisimple if and only if $\Delta(x) \neq 0$. Furthermore, we have that $x$ is regular semi-simple if and only if the $G$-orbit of $x$ is closed in $Y$ and $\Stab_G(x)$ is finite.
		\end{enumerate}
	\end{theorem}

	\begin{theorem}[{\cite[Theorem 1.10]{thorne2016arithmetic}}] \label{thm: thorne1.10}
		\
		\begin{enumerate}[(a)]
			\item
				Let $H$ be a split adjoint simple group of type $A$,$D$, or $E$ and let $\mathfrak{h}$ be its Lie algebra. Let $V := \mathfrak{h}^{d\theta = -1}$ and let $G := (H^\theta)^\circ$. Let $\mathfrak{t}$ be a Cartan subalgebra of $V$ and let $W_\mathfrak{t}$ be the Weyl group. Then the inclusion $\mathfrak{t} \ssq V$ induces an isomorphism
				\[
					\mathfrak{t} \sslash W_\mathfrak{t} \iso V \sslash G.
				\]
				Additionally, $V \sslash G$ is isomorphic to affine space.
			\item 
				Suppose that $k=k^\sep$ and let $x,y$ be regular semisimple elements. Then $x$ is $G(k)$-conjugate to $y$ if and only if $x,y$ have the same image in $V \sslash G$.
			\item 
				There exists a discriminant polynomial $\Delta \in k[V]$ such that for all $x \in Y$, $x$ is regular semisimple if and only if $\Delta(x) \neq 0$. Furthermore, we have that $x$ is regular semisimple if and only if the $G$-orbit of $x$ is closed in $V$ and $\Stab_G(x)$ is finite.
			\item
				Moreover, we have that the isomorphism from part (a) induces isomorphisms
				\[
					\xym{
						\mbp \mathfrak{t} \sslash W_\mathfrak{t}  \ar@{}[d]|*=0[@]{\cong} \ar@{}[r]|*=0[@]{\cong} & \mbp V \sslash G \ar@{}[d]|*=0[@]{\cong} \\
						(\mathfrak{t} \bs \{0\}) \sslash (W_\mathfrak{t} \times \mbg_m) \ar@{}[r]|*=0[@]{\cong} & (V \bs \{0\}) \sslash (G \times \mbg_m) .
					}
				\]
		\end{enumerate}		
	\end{theorem}

	\begin{proof}
		The statements of parts (a-c) are directly from \cite[Theorem 1.10]{thorne2016arithmetic}, so we need only prove the last claim.
		
		By definition, we have that $\mbp V := (V \bs \{0\}) \sslash \mbg_m$, where $\mbg_m$ acts on $V$ via $\lambda g \mapsto (\lambda \cdot \id_V) (g)$. In particular, $\lambda$ acts on $V$ through the centre of $\GL(V)$, so we ascertain that $G \times \mbg_m$ acts on $V$. Moreover, we have that $G \times \mbg_m$ acts on $\mathfrak{t}$ via the inclusion $i\: \mathfrak{t} \inj V$. The action of $G \times \id_{\mbg_m}$ on $\mathfrak{t}$ factors through $W_\mathfrak{t}$ by the first part of the theorem, and the action of $\id_G \times \mbg_m$ on $\mathfrak{t}$ factors faithfully through the centre of $\GL(\mathfrak{t})$ since $i$ is a morphism of Lie algebras. Thus, we see that $W_\mathfrak{t} \times \mbg_m$ acts on $\mathfrak{t}$ and that this action is compatible with the action of $G \times \mbg_m$ on $V$. 
		
		From part (a), we have that $i$ induces an isomorphism $i\: \mathfrak{t} \sslash W_\mathfrak{t} \iso V \sslash G$. Since this induced isomorphism is also $\mbg_m$-equivariant, we have
			\begin{align*}
				& ((\mathfrak{t} \bs \{ 0 \}) \sslash \mbg_m) \sslash W_\mathfrak{t} \iso
				(\mathfrak{t} \bs \{0\}) \sslash ( W_\mathfrak{t} \times \mbg_m) \\ 
				\iso \ \ & ((\mathfrak{t} \bs \{0\}) \sslash W_\mathfrak{t} ) \sslash \mbg_m 
				\overset{i}{\iso} ((V \bs \{0\}) \sslash G) \sslash \mbg_m \\
				\iso \ \ & (V \bs \{0\}) \sslash (G \times \mbg_m) \iso
				((V \bs \{0\}) \sslash \mbg_m) \sslash G 
			\end{align*}
		which gives the result.	
	\end{proof}	
	
	The approach we use in this article requires several results regarding involutions of a split adjoint simple algebraic group. Historically, Vinberg theory has been of great help in understanding these involutions. The specific involutions we are interested in are the stable involutions (in the sense of \cite[Section 2]{thorne2013vinberg}) which are also split.
	
	\begin{proposition}[{\cite[Proposition 1.9]{thorne2016arithmetic}}] \label{prop: thorne involution}
		Let $H$ be a split adjoint simple group of type $A$, $D$, or $E$ over $k$. There exists a unique $H(k)$-conjugacy class of involutions $\theta$ of $H$ satisfying the following two conditions:
			\begin{enumerate}[(i)]
				\item 
					$\trace( d\theta\: \mathfrak{h} \ra \mathfrak{h}) = -\rank H$
				\item
					The group $(H^\theta)^\circ$ is split.
			\end{enumerate}
	\end{proposition}
	
	It will be convenient to name the involutions satisfying the conditions of Proposition \ref{prop: thorne involution}.
	
	\begin{definition}
		Let $H$ be a split adjoint simple group of type $A$, $D$, or $E$ over $k$. We call an involution $\theta \: H \ra H$ satisfying the conditions of Proposition \ref{prop: thorne involution} an involution \emph{of Thorne type}.
	\end{definition}

\subsection{The refined construction of Lurie} \label{sec: Lurie}

We give a short summary of \cite[Section 2]{thorne2016arithmetic} as this will be helpful to state the main result we rely on. For details the reader is encouraged to consult the original article.

\pagebreak

Consider the following collection of abstract data: 

\bigskip
\noindent
\textbf{Data I:}		
\begin{enumerate}
	\item 
	An irreducible simply laced root lattice $(\Lambda, \gen{\cdot, \cdot})$ together with a continuous homomorphism $G_k \ra W_\Lambda \ssq \Aut(\Lambda)$. Here $W_\Lambda$ is the Weyl group of $\Lambda$.
	\item
	A central extension $\wtilde V$ of $V := \Lambda/2\Lambda$, where $\wtilde V$ satisfies
	\[
	\exactseq{\{\pm 1\}}{\wtilde V}{V}
	\]
	and for any $\wtilde v \in \wtilde V$ we have that $\wtilde v^2 = (-1)^{\frac{\gen{v,v}}{2}}$. 
	\item
	A continuous homomorphism $G_k \ra \Aut(\wtilde V)$ leaving $\{\pm 1 \}$ invariant, compatible with $G_k \ra \Aut(\Lambda) \ra \Aut(V)$.
	\item
	A finite dimensional $k$-vector space $W$ and a homomorphism $\rho\: \wtilde V \ra \GL(W)$ of $k$-groups.  
\end{enumerate}

We denote the ensemble of such data by a quadruplet $(\Lambda, \wtilde V, W, \rho)$. We may form the category $\mathcal{C}(k)$ of quadruplets satisfying conditions (1)-(4), whose morphisms are the obvious morphisms. Thorne \cite[Section 2]{thorne2016arithmetic} presents a modified construction of Lurie \cite{lurie2001simply}, which given a quadruplet $(\Lambda, \wtilde V, W, \rho)$ produces:

\bigskip
\noindent
\textbf{Data II:}	
\begin{enumerate}[(a)]
	\item 
	A simple Lie algebra $\mathfrak{h}$ over $k$ of type equal to the Dynkin type of $\Lambda$.
	\item
	A Lie group $H$, which is the adjoint group over $k$ whose Lie algebra is $\mathfrak{h}$, and a maximal torus $T$ of $H$. The torus $T$ is canonically isomorphic to $\Hom(\Lambda, \mbg_m)$.
	\item
	An involution $\theta\: H \ra H$ which acts on $T$ by $\theta(t) = t^{-1}$.
	\item
	An isomorphism $T[2](k^\sep) \iso V^\vee$ of Galois modules.
	\item
	A Lie algebra homomorphism $\rho\: \mathfrak{h}^{d\theta=1} \ra \mathfrak{gl}(W)$ defined over $k$.
\end{enumerate}

We denote the ensemble of such data by a tuple $(H, \theta, T, \rho)$, and we may form a category $\mathcal{D}(k)$ whose objects are tuples satisfying conditions (a)-(e) and whose morphisms are morphisms of the underlying Lie groups $\varphi\: H_1 \ra H_2$ over $k$ such that $\varphi$ intertwines the involutions $\theta$ (i.e, $\theta_2 \varphi = \varphi \theta_1$), satisfies $\rho_2 \circ d\varphi_e = \rho_1$, and such that $\varphi(T_1) \ssq T_2$. Note that while we require $\varphi(T_1) \ssq T_2$, this inclusion is not necessarily canonically defined. We point out that in Lemma \ref{lem: Lurie-and-descent}, we construct morphisms in $\mathcal{D}(k)$ where the tori \emph{are} canonically identified. We phrase the construction of Lurie as modified by Thorne in the following way:

\begin{theorem} \label{thm: Lurie construction}
	For each field $k$ of characteristic $0$, there is a functor $f_k\: \mathcal{C}(k) \ra \mathcal{D}(k)$ which is injective on objects.
\end{theorem}

\begin{remark}
	One can infer that this functor is injective on objects from the appendix of \cite{thorne2016arithmetic}, due to Kaletha. 
\end{remark}

\section{Uniquely trigonal genus $4$ curves, del Pezzo surfaces of degree $1$, and $E_8$} \label{sec: uniquely trigonal locus}

	We now describe the classically understood relation between the $E_8$ lattice, del Pezzo surfaces of degree $1$, and uniquely trigonal genus $4$ curves. We point out that \cite{thorne2016arithmetic} relies on the analogous connections between plane quartic curves, del Pezzo surfaces of degree $2$, and the $E_7$ lattice.

	Generally, a genus 4 curve is trigonal in two different ways, and the corresponding linear systems of divisors can be found as follows. A canonical model of a non-hyperelliptic genus 4 curve is a complete intersection of a quadric and a cubic in $\mbp^3$. A linear system of lines on the quadric induces a linear system of degree 3 and dimension 1 on the curve. In the general case, the quadric is nonsingular and has two such linear systems. In the special case where the quadric has a singular point, there is only one such linear system. In that case the curve is trigonal in only one way.

	\begin{definition} \label{def: uniquely trigonal}
		Let $C/k$ be a smooth curve of genus $g$. We say that $C$ is a \emph{uniquely trigonal} curve if for any two morphisms $\pi, \pi'\: C_{k^\sep} \ra \mbp_{k^\sep}^1$ of degree $3$, there is a $\tau \in \Aut(\mbp_{k^\sep}^1)$ such that $\pi = \tau \circ \pi'$.
	\end{definition}

	\begin{remark}
		It is impossible for a smooth curve of genus greater than $2$, defined over a characteristic $0$ field, to be both trigonal and hyperelliptic \cite[Section 2.8]{vakil2001twelve}.
	\end{remark}

	Let $S/k$ be a del Pezzo surface of degree $1$ and let $\mathcal{K}_S$ be the canonical divisor of $S$. We have that $-3\mathcal{K}_S$ is very ample, and the model of $S$ given by the associated linear system embeds $S$ as a smooth sextic hypersurface in the weighted projective space $\mbp(1:1:2:3)$. Conversely, any smooth sextic hypersurface in $\mbp(1:1:2:3)$ is a del Pezzo surface of degree $1$. Additionally, the linear system associated to the divisor $-2\mathcal{K}_S$ determines a rational map $S \dashrightarrow \mbp(1:1:2)$ of generic degree $2$. There is an order $2$ automorphism of $S$ which exchanges the branches of this rational map called the \emph{Bertini involution}. Moreover, the fixed locus of the Bertini involution is the union of a smooth irreducible curve and an isolated point. We call the component which is the smooth irreducible curve the \emph{branch curve}. For details, see \cite[Section 8.8.2]{dolgachev2012classical}.
	
	Let $C$ be the branch curve of the Bertini involution. The linear system associated to the divisor $-\mathcal{K}_S$ determines a rational map $\pi_S\: S \dashrightarrow \mbp^1$ which restricts to a degree $3$ morphism $\pi_C\: C \longrightarrow \mbp^1$. The sections of the anti-canonical bundle of $S$ are precisely the fibres of $\pi_S$. It is classically known that the set of $p \in \mbp^1(k^\sep)$ such that the curve $\pi_S^{-1}(p)$ is singular is exactly the set of $p \in \mbp^1(k^\sep)$ such that there is a ramification point of $\pi_C$ lying over $p$ (see \cite[Corollary 2.2]{kulkarni2016explicit}). For details, see \cite[Section 8.8.3]{dolgachev2012classical}. Theorem \ref{thm: uniquely trigonal genus 4} is a compilation of classical facts.
		
	\begin{theorem} \label{thm: uniquely trigonal genus 4}
		Let $C$ be a curve of genus $4$ which is not hyperelliptic. Then the following are equivalent:
			\begin{enumerate}[(a)]
				\item 
					The curve $C$ is a uniquely trigonal curve.
				\item
					The canonical model of $C$ is the intersection of a cubic and a quadric cone in $\mbp^3$.
				\item
					There is a vanishing even theta characteristic of $C$. That is, there is an isomorphism class of line bundles on $C_{k^\sep}$ such that for any member $\mathcal{L}_\theta$, we have that $\mathcal{L}_\theta^{\otimes 2}$ is isomorphic to the canonical bundle and $h^0(C_{k^\sep}, \mathcal{L}_\theta)$ is a positive even integer.
				\item
					There is a unique vanishing even theta characteristic of $C$.
				\item
					There is a del Pezzo surface $S$ of degree $1$ defined over $k$ such that $C$ is the branch curve of the Bertini involution on $S$. If $k = k^\sep$ then $S$ is unique up to isomorphism.
			\end{enumerate}
	\end{theorem}

	\begin{proof}
		Parts (a), (b), (c), and (e) can be found in  \cite[Section 2.8]{vakil2001twelve} and \cite[Proposition 3.2]{vakil2001twelve}. Part (c) follows trivially from part (d). To show (c) implies (d), if $\mathcal{L}_{\theta_1}$ and $\mathcal{L}_{\theta_2}$ are two line bundles corresponding to vanishing theta characteristics, then $h^0(C_{k^\sep}, \mathcal{L}_{\theta_1} \otimes \mathcal{L}_{\theta_2}) \geq 4$. The Riemann-Roch theorem shows that $\mathcal{L}_{\theta_1} \otimes \mathcal{L}_{\theta_2}$ is equivalent to the canonical bundle, so it follows after a simple calculation that $\mathcal{L}_{\theta_1} \iso \mathcal{L}_{\theta_2}$.  
	\end{proof}

	Let $S/k$ be a del Pezzo surface of degree $1$. It is a classical fact that $\Pic S_{k^\sep} \iso \mbz^9$ as a group, and that the intersection pairing $\gen{\cdot, \cdot}$ on $S$ imbues $\Pic S_{k^\sep}$ with the structure of a lattice. As the canonical class $\mathcal{K}_S$ is always defined over $k$, the sublattice
		\[
			(\Pic S_{k^\sep})^\perp := \{x \in \Pic S_{k^\sep} : \gen{x , \mathcal{K}_S} = 0 \}
		\] 
	is also defined over $k$. Moreover, it is a classical fact that $(\Pic S_{k^\sep})^\perp$ is isomorphic to a simply laced lattice of Dynkin type $E_8$. We denote the Weyl group of the lattice of Dynkin type $E_8$ by $W_{E_8}$ and note that $W_{E_8}$ is also the isometry group for this lattice. Additionally, the Bertini involution of $S$ acts on $(\Pic S_{k^\sep})^\perp$ via the unique nontrivial element of the centre of $W_{E_8}$. We label the elements of the centre of $W_{E_8}$ by $\pm 1$ and we let $W_+ := W_{E_8}/\gen{\pm 1}$. For details, see \cite[Section 2.2]{zarhin2008delpezzo}.
	
	\begin{definition}
		A \emph{marked del Pezzo surface} of degree $1$ is a pair $(S/k, B)$ with $S/k$ a del Pezzo surface of degree $1$ and $B = \{e_1, \ldots, e_8 \}$ a basis for $\Pic(S_{k^\sep})^\perp$ such that $\gen{e_i, e_i} = -1$ for each $i$ and $\gen{e_i, e_j}=0$ for all $i \neq j$. We refer to the choice of basis $B$ as a \emph{marking}.
	\end{definition}
	
	We will make use of the following two classical results regarding degree $1$ del Pezzo surfaces. These are of course analogous to the results used in ~\cite{thorne2016arithmetic} for degree $2$ del Pezzo surfaces. 
	
	\begin{proposition} \label{prop: del Pezzo 1}
		Let $C$ be the branch curve of the Bertini involution of a degree $1$ del Pezzo surface $S/k$ and let $\Lambda := (\Pic S_{k^\sep})^\perp$. Then there is a commutative diagram of finite $k$-groups
		\[
			\xym{
				\Lambda^\vee/2\Lambda^\vee \ar[r]^-{\iso} & (\Pic(C)/\mbz \theta_C)[2] \\
				N_C \ar@{^{(}->}[u] \ar[r]^-\iso & \Pic^0(C)[2] \ar@{^{(}->}[u] 
			}
		\]
		where $N_C$ is the image of 
			\mapdef{\gamma}{\Lambda/2\Lambda}{\Lambda^\vee/2\Lambda^\vee}{v}{\gen{v, \cdot}}
		and $\theta_C$ is the divisor class of the vanishing even theta characteristic of $C$. In particular, there is a canonical surjection $\Lambda/2\Lambda \ra \Pic^0(C)[2]$.
	\end{proposition}
	
	\begin{proposition} \label{prop: del Pezzo 2}
		Let $\phi\: \Lambda/2\Lambda \ra \Pic^0(C)[2]$ be the surjection of Proposition \ref{prop: del Pezzo 1}. Let $\gen{\cdot, \cdot}_2$ be the natural symplectic form on $\Lambda/2\Lambda$ induced by the form $\gen{\cdot, \cdot}$ on the even lattice $\Lambda$ and let $e_2(\cdot, \cdot)$ be the Weil pairing on $\Pic^0(C)[2]$. Then for all $v_1, v_2 \in \Lambda/2\Lambda$ we have
		\[
			\gen{ v_1, v_2 }_2 = e_2( \phi v_1, \phi v_2).	
		\]
	\end{proposition}
	
	For the proofs of Proposition \ref{prop: del Pezzo 1} and Proposition \ref{prop: del Pezzo 2}, see \cite[Lemma 2.4]{zarhin2008delpezzo}, \cite[Lemma 2.7]{zarhin2008delpezzo}, and \cite[Theorem 2.10]{zarhin2008delpezzo}. From these propositions we see that a marking of a degree $1$ del Pezzo surface, up to $\pm 1$, uniquely determines a marking of the $2$-torsion of its branch curve.

\subsection{Points on maximal tori in $E_8$}
	
	Theorem \ref{thm: uniquely trigonal genus 4} provides a relationship between the uniquely trigonal genus $4$ curves and del Pezzo surfaces of degree $1$. Similar to the relationship between plane quartic curves and degree $2$ del Pezzo surfaces in \cite{looijenga1993cohomology}, we can ask what the datum of an anti-canonical section of a degree $1$ del Pezzo surface marks on the associated uniquely trigonal genus $4$ curve. Corollary \ref{cor: e8 3.4} provides the answer to this question.
	
	Before progressing, it is useful to have an explicit description for the root lattice $\Lambda$ of type $E_8$. We use the following description which is from \cite[Section 8.2.2]{dolgachev2012classical}. Let $\{l, e_1, \ldots, e_8 \}$ be the free generators for a $9$-dimensional lattice where the generators are pairwise orthogonal, $\gen{e_i , e_i} = -1$ for each $i$, and $ \gen{l , l} = 1$. One recovers $\Lambda$ as the sublattice spanned by elements satisfying $\gen{\alpha , \alpha} = -2$. The roots of $\Lambda$ are exactly those elements such that $\gen{\alpha , \alpha} = -2$.  
	
	\begin{remark}
		Let $\Lambda$ be a root lattice of type $E_8$ and let $T = \Hom_k(\Lambda,\mbg_m)$ be a torus. Then $\chi \in T^\rss$ if and only if $\chi$ does not lie on a root hyperplane, which is equivalent to $\chi(\alpha) \neq 1$ for every root $\alpha$ \cite[Section 1]{looijenga1993cohomology}.  
	\end{remark}

	We prove an analogue of ~\cite[Proposition 1.8]{looijenga1993cohomology} for degree $1$ del Pezzo surfaces. We let $\wtilde \mcDP_1$ denote the moduli space of marked degree 1 del Pezzo surfaces and let $\wtilde \mcDP_1(\node)$ denote the moduli space of marked degree 1 del Pezzo surfaces together with a singular nodal anti-canonical section.

	\begin{proposition} \label{prop: torusE8}
		If $k=k^\sep$, then there is a $W_+$-equivariant isomorphism
			\[
				\wtilde \mcDP_1(\node) \iso T^{\rss} / \gen{\pm 1}
			\]
		with $T := \Hom(\Lambda, \Gm)$ and $\Lambda$ a split root lattice of type $E_8$.
	\end{proposition}

	\begin{proof}
		Let $K$ be an abstract nodal genus $0$ curve. Then $\Pic^0(K)$ is isomorphic to $\mbg_m$ in a unique way up to inversion. Fix a choice of $P_1 \in K_{\mathrm{reg}}$ and corresponding isomorphism $\tau_{P_1}\: \Pic^0(K) \ra \Pic^1(K) \iso K_{\reg}$ given by translation by $P_1$. 
		
		\bigskip
		Let $\chi \in T^{\mathrm{rss}}$ be a point of the torus $T = \Hom(\Lambda, \mbg_m)$. For each $i \in \{1, \ldots, 7\}$ we define $P_{i+1}$ to be the unique point of $K_\reg(k^\sep)$ such that the divisor $P_{i+1}$ is linearly equivalent to $P_i + \chi(e_{i+1}-e_i)$. Notice that for any $i,j \in \{1, \ldots, 8\}$ we have that $(P_i - P_j) = \chi(e_i - e_j)$. 

		The linear system associated to the degree $3$ divisor
			\[
				D := \chi(l - e_1 - e_2 - e_3) + P_1 + P_2 + P_3
			\] 
		determines an embedding of $K$ into the projective plane. We claim under this embedding that the points $P_1, \ldots, P_8$ are in general position. 
		
		\bigskip
		
		We see that two points coincide if and only if $\chi(e_i - e_j) = \id$ for some distinct $i,j$. We see that $P_{i_1}, P_{i_2}, P_{i_3}$ lie on a line if and only if we have
			\begin{align*}
				\div h &=  D - (P_{i_1} + P_{i_2} + P_{i_3}) \\
					&= \chi(l - e_1 - e_2 - e_3) + P_1 - P_{i_1} + P_2 - P_{i_2} + P_3 - P_{i_3} \\
					&= \chi(l - e_1 - e_2 - e_3) + \chi(e_{i_1} - e_1) + \chi(e_{i_2} - e_1) + \chi(e_{i_3} - e_1) \\
					&= \chi(l - e_{i_1} - e_{i_2} - e_{i_3})
			\end{align*}
		for some $h \in k(K)$. Similar calculations show that six of these points lie on a conic if and only if $\chi(2l - e_{i_1} - \ldots - e_{i_6}) = \id$ for some distinct $i_j \in \{1, \ldots, 8\}$ and there is a cubic passing through all eight of these points with a singularity at one of them if and only if $\chi(3l - e_{i_1} - \ldots - e_{i_7} - 2e_{i_8}) = \id$ for some distinct $i_j \in \{1, \ldots, 8\}$. To recapitulate, the points $P_1, \ldots, P_8$ lie in general position if and only if $\chi(\alpha) \neq 1$ for each root $\alpha$ of $\Lambda$. 
		
		\bigskip
		
		The blow-up of $\mbp^2$ at the eight points $P_1, \ldots, P_8$ is a marked degree 1 del Pezzo surface with a marked nodal anti-canonical curve $(S, \{e_1, \ldots, e_8 \}, K')$, where $K'$ is the strict transform of $K$ under the blow-up. 
		
		If the construction above sends $\chi$ to $(S, \{e_1, \ldots, e_8 \}, K')$, then it sends $-\chi$ to $(S, \{\iota(e_1), \ldots, \iota(e_8)\}, K')$, where $\iota$ is the Bertini involution of $S$. It is clear that every marked degree 1 del Pezzo surface with a marked nodal section is obtained by this construction in a unique way up to inversion on $T$. Additionally, this map is $W_{E_8}$-equivariant, so we obtain a $W_+$-equivariant map $T/\gen{\pm 1} \ra \wtilde \mcDP_1(\node)$.
	\end{proof}

	\begin{corollary} \label{cor: e8 3.4}
		There is a $W_+$-equivariant isomorphism from $T^\rss$ to $\mathcal{T}_4^{\sram}(2)$.
	\end{corollary}

	\begin{proof}
		Proposition \ref{prop: torusE8} supplies a $W_+$-equivariant isomorphism \break $T^\rss \iso \wtilde \mcDP_1(\node)$. However, there is a $W_+$-equivariant isomorphism from $ \wtilde \mcDP_1(\node)$ to $\mathcal{T}_4^{\ram}(2)$. Explicitly, if $(S,s)$ is a marked degree $1$ del Pezzo surface, we obtain a uniquely trigonal genus $4$ curve with marked $2$-level structure by taking the branch curve. By ~\cite[Corollary 2.2]{kulkarni2016explicit}, the section $s$ corresponds to a unique ramification point of index $2$. 
	\end{proof}
	
	If $K$ is an abstract cuspidal genus $0$ curve, then we may choose an isomorphism $\mbg_a \iso \Pic^0(K)$. This choice of isomorphism is unique up $\Aut(\mbg_a) = \mbg_m$, and the $-1$ element of $\mbg_m$ acts on $\mbg_a$ by inversion. Thus, if $\Lambda$ is the character lattice of a torus $T$, then $\Aut(\mbg_a)$ acts on $\mathfrak{t} := \Hom(\Lambda, \mbg_a)$ and we have $(\mathfrak{t} \bs \{0\}) \sslash\Aut(\mbg_a) = \mbp \mathfrak{t}$.
	
	Let $\wtilde \mcDP_1(\cusp)$ denote the moduli space of marked degree 1 del Pezzo surfaces together with a singular cuspidal anti-canonical section. By replacing the singular nodal cubic in the proof of Proposition \ref{prop: torusE8} with a cuspidal plane cubic we can prove using an identical argument:
	
	\begin{proposition} \label{prop: addtorusE8}
		If $k=k^\sep$, there is a $W_+$-equivariant isomorphism
		\[
		\wtilde \mcDP_1(\cusp) \iso  \mbp \mathfrak{t}^{\rss}
		\]
		with $\mathfrak{t} := \Hom(\Lambda, \Ga)$ and $\Lambda$ a split root lattice of type $E_8$.
	\end{proposition}

	In the proposition above, one can think of $\mathfrak{t}$ as the Lie algebra of a maximal subtorus of a split adjoint simple group of type $E_8$. As before we also obtain:

	\begin{corollary} \label{cor: e8lie 3.4}
		There is a $W_+$-equivariant isomorphism from $\mbp \mathfrak{t}^\rss$ to $\mathcal{T}_4^{\tram}(2)$.
	\end{corollary}

	\begin{remark} \label{rem: assignment-map}
		Let $\Lambda$ be the simply laced lattice of Dynkin type $E_8$. Let $C/k$ be a uniquely trigonal genus $4$ curve, let $\pi\: C \ra \mbp^1$ be the unique up to $\PGL_2$ trigonal morphism, and let $P$ be a ramified point of $\pi$. We give the details of the assignment of $(C,P) \in \mathcal{T}_4^{\sram}(k)$ to a rational point on the torus $T = \Hom(\Lambda, \mbg_m)(k)$ which is well-defined up to $W_{E_8} \iso \Aut(\Lambda)$.
		
		\bigskip
		
		The canonical model of $C$ lies on a quadric cone in $\mbp^3$. We may identify the quadric cone with $\mbp(1:1:2)$ and up to automorphisms of $\mbp^3$ we see that $C$ is given by a model of the form
		\[
			C\: 0 = f_0 w^3 + f_2(s,t) w^2 + f_4(s,t) w + f_6(s,t).
		\]
		Since $C$ is not hyperelliptic, we have that $f_0 \neq 0$. The uniquely trigonal morphism $\pi$ is induced by the projection $\pi\: (s,t,w) \mapsto (s,t)$. We let $S/k$ be the degree $1$ del Pezzo surface defined by the sextic equation in $\mbp(1:1:2:3)$
		\[
			S\: z^2 = f_0 w^3 + f_2(s,t) w^2 + f_4(s,t) w + f_6(s,t).
		\]
		Using the marked point $P = (s_0,t_0,w_0)$ on $C$ we define an anti-canonical section on $S$. The map 
			\mapdef{\pi}{X}{\mbp^1}{(s,t,w,z)}{(s,t)}
		is a rational map which is regular outside the base-point of $S$. The Weil divisor $D = \overline{ \pi^{-1}(s_0,t_0) }$ is an anti-canonical divisor of $S$. Moreover, by the assumption that $P$ is a simply ramified point of $\pi$, we see that $D$ as a scheme is isomorphic to a nodal genus $0$ curve. The restriction $\Pic(S_{k^\sep}) \ra \Pic(D)$
		induces a homomorphism of Galois modules
		\[
			\Pic(S_{k^\sep})^\perp \ra \Pic^0(D). 
		\]
		Up to $\Aut(\Lambda)$ we may choose an identification $\Lambda \iso \Pic(S_{k^\sep})^\perp$. We may also choose an identification $\Pic^0(D) \iso \mbg_m$, and the choice of this identification is unique up to inversion. We obtain a Galois action on $\Lambda$ by inheriting the Galois action on $\Pic(S_{k^\sep})^\perp$. Thus, we obtain a point $\kappa_C \in \Hom(\Lambda, \mbg_m)(k)$ and this point is unambiguously defined up to $\Aut(\Lambda)$. We produce a similar assignment $\mathcal{T}_4^{\tram}(k) \dashrightarrow \mbp \Hom(\Lambda, \mbg_a)(k)$ in an analogous fashion.
		
	\end{remark}
	
	Via Corollary \ref{cor: e8 3.4}, Corollary \ref{cor: e8lie 3.4}, and Remark \ref{rem: assignment-map} we obtain:
	
	\begin{reptheorem}{conj: e8 3.4}
		\
		\begin{enumerate}
			\item[$E_8\:$]
			If $k=k^\sep$, then the assignment of Remark \ref{rem: assignment-map} induces a bijection
			\[
			\mcT^\sram(k) \ra (T^\rss \sslash W_T)(k)
			\]
			with $\Lambda$ the root lattice of type $E_8$ and $T = \Hom(\Lambda, \Gm) $ the split torus of rank $8$.
			
			\item[$\mathfrak{e}_8\:$]
			If $k=k^\sep$, then the assignment of Remark \ref{rem: assignment-map} induces a bijection
			\[
			\mcT^\tram(k) \ra (\mbp \mathfrak{t}^\rss \sslash W_T)(k)
			\]
			with $\mathfrak{t}$ the Lie algebra of the torus $T$.
		\end{enumerate}
	\end{reptheorem}

	\begin{remark}
		An additional advantage of using the space $\mbp \mathfrak{t}$ instead of $\mathfrak{t}$ is pointed out in \cite[Section 1.16]{looijenga1993cohomology}. Namely, that there is a natural way to fit together the isomorphisms from Proposition \ref{prop: torusE8} and Proposition \ref{prop: addtorusE8}.
	\end{remark}

\section{Representation data and $G^\circ$-conjugacy classes} \label{sec: Reeder calculation}

	In this section, we perform a minor but important technical calculation using the method of ~\cite{reeder2010torsion} to compute a certain component group. The main results from this section are Corollary \ref{cor: component counting}, Lemma \ref{lem: representation dimension}, and Lemma \ref{lem: h^theta is D8}. 

	Let $H$ be a split adjoint simple group over $k$ of type $E_8$. For an involution $\theta$ we let $G := (H^\theta)^\circ$. By Proposition \ref{prop: thorne involution}, if $\theta_0, \theta$ are involutions of $H$ of Thorne type, there is a unique $\phi \in \Aut(H)$ up to $H^\theta(k)$-conjugacy such that $\phi \theta_0 \phi^{-1} = \theta$. However, we will require a way to determine $\phi$ uniquely up to $G(k)$-conjugacy. \cite{reeder2010torsion} shows that discrepancy is explained by a finite group of connected components. We calculate this component group in Corollary \ref{cor: component counting} and show that it is trivial. Nevertheless, it is still useful to supply the extra data of a representation, as in ~\cite{thorne2016arithmetic}, to resolve other technicalities. 

\subsection{Preliminaries for root systems}

	Let $\mathfrak{h}$ be the Lie algebra of $H$ and let $\mathfrak{t}$ be the Lie algebra of a maximal split subtorus $T$. In other words, $\mathfrak{t}$ is a Cartan subalgebra of $\mathfrak{h}$. Let $\mathfrak{t}^\vee := \Hom_k(\mathfrak{t},k)$ denote the linear dual to $\mathfrak{t}$. For any $\alpha \in \mathfrak{t}^\vee$, we let
		\[
			\mathfrak{h}^\alpha := \{ x \in \mathfrak{h} : (\ad_\mathfrak{h} t) x = \alpha(t) x \textbf{ for all } t \in \mathfrak{t}  \}.
		\]
	There are only finitely many $\alpha$ for which $\mathfrak{h}^\alpha \neq 0$; these $\alpha$ are called the \emph{roots} of $(\mathfrak{h}, \mathfrak{t})$ \cite[Chapter 8, Section 2]{bourbaki1975lie}.
	
	Let $\Phi$ denote the set of roots of $(\mathfrak{h}, \mathfrak{t})$. We see that the $\mbz$-linear combinations of elements of $\Phi$, which we denote by $\mbz \Phi$, form a lattice called the \emph{root lattice} of $(\mathfrak{h},\mathfrak{t})$. We see that $\mbz \Phi$ is also (by definition) the character lattice of $\mathfrak{t}$. The character lattice of $T$ can be canonically identified with the character lattice of $\mathfrak{t}$, so we write $\Lambda = \mbz \Phi$.
	
\subsection{Computing the connected components of $H^\theta$}

	We compute the Zariski connected components of $H^\theta$ using the method of \cite{reeder2010torsion}, on which we give an exposition below. We remark that it suffices to determine the connected components in the analytic topology as these are exactly the connected components in the Zariski topology.

		Let $\Lambda$ be the character lattice of $T$ and let $Y = \Hom(\Lambda, \mbz)$ be the co-character lattice of $T$.
	Let $V := Y \otimes_\mbz \mbr$. We may regard $V$ as the Lie algebra of the maximal (analytically) compact subtorus $S \ssq T$ via the exponential map
		\[
			\exp\: V \ra S,
		\]
	which is a surjective group homomorphism. The kernel of $\exp$ is exactly $Y$, so in particular we have an isomorphism
		\[
			\exp\: V/Y \overset{\sim}{\longrightarrow} S.
		\]
	The action of the Weyl group $W_T$ on $\Lambda$ gives rise to a dual action on $Y$ and hence on $V$. This action is faithfully represented by the action of the reflection group of the coroot lattice $\mbz \Phi^\vee \ssq Y$ extended to $V$. Note that in general $\mbz\Phi^\vee \neq Y$, though we do have equality when $\Lambda$ is a root lattice of type $E_8$. This fact is not true for $E_7$, and we shall discuss exactly how this difference manifests in terms of invariant theory in Remark \ref{rem: E7 versus E8}.

	We now recall some facts stated in ~\cite[Section 2.2]{reeder2010torsion}. Any element of $H$ which acts on $\Lie H$ diagonalizably (over $\mbc$) is $H$-conjugate to an element of $T$. Additionally, we have that two elements of $T$ are $H$-conjugate if and only if they are conjugate by $W_T$. Any torsion element of $H$ acts diagonalizably over $\mbc$ and is conjugate to an element of $S$. The elements $s = \exp(x)$ and $s' = \exp(x')$ of $S$ are conjugate if and only if $x,x'$ are in the same orbit under action of the \emph{extended affine Weyl group}
		\[
			\wtilde W_T := W_T \ltimes Y
		\]
	where $Y$ acts on $V$ by translations. 
	
	Let $\Delta := \{\alpha_1, \ldots, \alpha_\ell\}$ be a basis of simple roots of $\Lambda$ and let $\{\check \omega_1, \ldots, \check \omega_\ell \}$ be the basis of the co-weight lattice dual to $\Delta$. Which is to say, under the natural pairing we have that $\gen{\alpha_i, \check \omega_j } = \delta_{ij}$. Let $\alpha_0 = \sum_{i=1}^\ell a_i \alpha_i$ be the highest root of $\Lambda$ with respect to $\Delta$ (here the $a_i$ are positive integers). Let $v_i := a_i^{-1} \check \omega_i$ for $1 \leq i \leq \ell$ and let $v_0 := 0$.
	
	The set
		\[
			C = \set{\sum_{i=0}^\ell x_i v_i \in V : 0 < x_i < 1  \text{ and } \sum_{i=0}^\ell x_i = 1 }
		\]
	is the \emph{alcove} determined by $\Delta$ (for a proper definition of alcove see ~\cite[Section 2.2]{reeder2010torsion}). For each $g \in \wtilde W_T$ we may write $g \cdot C := \{x \in V : g^{-1}(x) \in C \}$. Note that the closure of $C$ is the simplex of dimension $\ell$ given by the convex hull of $\{ v_0, \ldots, v_\ell\}$. It is a fact that $g \cdot C \cap C \neq \emptyset$ if and only if $g \cdot C = C$. An important group related to $H$ is the \emph{alcove stabilizer}, which is defined as
		\[
			\Omega := \{g \in \wtilde W_T : g \cdot C = C\}.
		\]
	For $x \in V$, we further define $\Omega_x$ to be the subgroup of $\Omega$ stabilizing $x$. It is computationally useful to note that $\Omega_b = \Omega$ if $b$ is the barycentre of $C$. One has that $\Omega \iso Y / \mbz \Phi^\vee$ and that $\Omega$ is isomorphic to the fundamental group of $H$ (see ~\cite[Section 2.2]{reeder2010torsion}). The following useful result appears as ~\cite[Proposition 2.1]{reeder2010torsion}.
	
	\begin{proposition}[Reeder] \label{prop: reeder2.1}
		Let $\bar C$ denote the closure of $C$ in $V$. Denote by $C_H(s)$ the centralizer of $s$ in $H$. For $s = \exp(x)$ with $x \in \bar C$, the component group of $C_H(s)$ is isomorphic to $\Omega_x$.
	\end{proposition}
	
	We now prove the main result of this section. 

	\begin{corollary} \label{cor: component counting}
		Let $H$ be a split adjoint simple group of type $E_8$ and $\theta\: H \ra H$ an involution. Then $H^\theta$ has a single connected component. In particular, $H^\theta = (H^\theta)^\circ$.
	\end{corollary}

	\begin{proof}
		Recall that $\Omega \iso Y/\mbz\Phi^\vee$ and that $Y = \mbz\Phi^\vee$ for root lattices of type $E_8$. Thus $\Omega$ is trivial, so by the preceding discussion we have that $H$ is simply connected. Since the Dynkin diagram of $E_8$ admits no automorphisms, we have that every automorphism of $H$ is inner (see \cite[Exercise 8.28]{fulton1991representation}, \cite[Proposition D.40]{fulton1991representation} ). Note that as $H$ has trivial centre, it is isomorphic to its group of inner automorphisms, so we may choose an element $s \in H(\mbc)$ such that $s = s^{-1}$ and $\theta(h) = s h s^{-1}$ for all $h \in H$. But $\Omega$ is trivial, so by Proposition \ref{prop: reeder2.1} the centralizer of any $2$-torsion element has a single connected component. Since $H^\theta = C_{H}(s)$ we are done. 
	\end{proof}

	\begin{remark} \label{rem: E7 versus E8}
		If $H$ is a split adjoint simple group of type $E_7$, then the alcove stabilizer has order $2$. Thorne uses exactly this computation to show $H^\theta$ has two connected components if $\theta$ is a split involution of trace $-\rank H$. It is mostly for this reason that his argument requires the extra data of a representation (see the bulleted remarks in Section 3A of ~\cite{thorne2016arithmetic}, as well as the proof of Theorem 3.5). 
	\end{remark}

	A consequence of Corollary \ref{cor: component counting} is that the data consisting of a point on a maximal torus and a split involution acting by inversion on this torus is sufficient to determine a unique $(H^\theta)^\circ$-orbit. In particular, we do not need to keep track of any representation data to identify the correct connected component of $H^\theta$.

\subsection{Two representations of dimension $16$}

	Here we state some well-known facts regarding two $16$-dimensional representations used in Section \ref{sec: Construction of orbits}. 
	
	\begin{lemma} \label{lem: representation dimension}
		Let $B$ be a genus $4$ curve with theta characteristic $\kappa$. Let $W_{3}$ be the theta divisor on $\Pic^{3} B$. Let $\Theta_\kappa := t_\kappa^* W_{3} \in \Div (\Pic^0 B)$ be the pullback of $W_{3}$ along translation by $\kappa$. Then
		\[
			h^0(\Pic^0 B, 2\Theta_\kappa) = 16.
		\]
	\end{lemma}
	
	\begin{proof}
		See ~\cite[Chapter IV, Section 8]{birkenhake2004complex}.
	\end{proof}

	The statement of Lemma \ref{lem: representation dimension} does not require any property of $B$ beyond the genus.

	\begin{lemma} \label{lem: h^theta is D8}
		Let $H$ be a split adjoint simple group of type $E_8$ defined over $k$ and let $\theta\: H \ra H$ be an involution of Thorne type. Then
			\begin{enumerate}[(a)]
				\item 
					The Lie algebra $\Lie H^\theta$ is of Dynkin type $D_8$. In particular there is an isomorphism
						\[
							\Lie H^\theta \iso \mathfrak{so}_{16}.	
						\]
				\item
					The Standard representation of $\Lie H^\theta$ is a $16$-dimensional representation.
			\end{enumerate}
	\end{lemma}

	\begin{proof}
		One need only determine the Dynkin type of $H^\theta$ as the remaining statements are standard facts. By functoriality of $\Lie$, we have that $\Lie H^\theta = \mathfrak{h}^{d\theta=1}$, where $\mathfrak{h}^{d\theta=1}$ denotes the $+1$-eigenspace of the involution $d\theta$ acting on $\mathfrak{h}$. Additionally, we have that $H \iso \Aut(\mathfrak{h})^\circ$ since $H$ is a split adjoint simple Lie group of type $E$.
		
		By Proposition \ref{prop: thorne involution} it suffices to demonstrate the claim for a particular involution of Thorne type. Let $\mathfrak{t}$ be a maximal split Cartan subalgebra of $\mathfrak{h}$ and write 
			\[
				\mathfrak{h} = \mathfrak{t} \oplus \bigoplus_{\alpha \in \Phi(\mathfrak{h}, \mathfrak{t})} \mathfrak{h}^\alpha
			\]
		with $\Phi(\mathfrak{h},\mathfrak{t})$ denoting the roots of $\mathfrak{h}$ with respect to $\mathfrak{t}$. We fix an identification of $\Phi(\mathfrak{h}, \mathfrak{t})$ with the $E_8$ root system in $\mbr^8$ (as in \cite[Planche 7]{bourbaki1968lie}) and let $\Phi_\mbz(\mathfrak{h}, \mathfrak{t}) \ssq \Phi(\mathfrak{h},\mathfrak{t})$ be the roots whose coordinates have integer entries. Define $d\theta$ by linearly extending
			\[
				d\theta(x) := 
					\begin{cases}
						\phantom{ -} \id &\textbf{ if } x \in \mathfrak{h}^\alpha,\  \alpha \in \Phi_\mbz(\mathfrak{h},\mathfrak{t}) \\
						\phantom{-}\id &\textbf{ if } x \in \mathfrak{t} \\
						-\id &\textbf{ if } x \in \mathfrak{h}^\alpha, \ \alpha \in \Phi(\mathfrak{h},\mathfrak{t}) \bs \Phi_\mbz(\mathfrak{h},\mathfrak{t})
					\end{cases}.
			\]
		We note that for $\alpha, \beta \in \Phi(\mathfrak{h},\mathfrak{t})$, we have that $\alpha + \beta$ has non-integer coordinates precisely when exactly one of $\alpha, \beta$ has non-integer coordinates. Using \cite[Excercise D.5]{fulton1991representation} it is a simple calculation to check that $d\theta$ preserves the Lie bracket. Thus $d\theta$ is a split involution of $\mathfrak{h}$, which corresponds to a split involution $\theta \in H$. Moreover, we see that $\trace(d\theta) = -\rank H$, so $\theta$ is of Thorne type. It is a standard fact that the Dynkin type of the root system $\Phi_\mbz(\mathfrak{h}, \mathfrak{t})$ is $D_8$.
	\end{proof}


\section{Construction of orbits for the $E_8$ case} \label{sec: Construction of orbits}

	In this section we prove the main results stated in Section \ref{subsec: main statements}.

	\subsection{Proof of Theorem \ref{conj: e8 3.5}}

		\begin{reptheorem}{conj: e8 3.5}
			\
			\begin{enumerate}
				\item[$E_8\:$] 
					Let $H$ be the split adjoint simple group of type $E_8$ over $k$, let $\theta$ be an involution of $H$ of Thorne type, and let $X := (H^{\theta(h)=h^{-1}})^\circ$ be the theta inverted subvariety. Let $G := (H^\theta)^\circ$. Let $X^\rss$ be the open subset of regular semi-simple elements. Then the assignment $(C,P) \mapsto \kappa_C$ of Theorem \ref{conj: e8 3.4}-$E_8$ determines a map
						\[
							\mcT^\sram(k) \ra G(k) \bs X^\rss(k).
						\]
					If $k=k^\sep$, then this map is a bijection.	
				\item[$\mathfrak{e}_8\:$] 
					Let $H$ be as above, let $X := \mathfrak{h}^{d\theta=-1}$ be the Lie algebra of $(H^{\theta(h)=h^{-1}})^\circ$, and let $X^\rss$ be the open subset of regular semi-simple elements. Then the assignment $(C,P) \mapsto \kappa_C$ of Theorem \ref{conj: e8 3.4}-$\mathfrak{e}_8$ determines a map
						\[
							\mcT^\tram(k) \ra G(k) \bs \mbp X^\rss(k).
						\]
					If $k=k^\sep$, then this map is a bijection.
			\end{enumerate}
		\end{reptheorem}
		
		\begin{remark} \label{rem: W_T ambiguity}
			We often make use of an assignment $(C,P) \dashrightarrow \Hom(\Lambda, \mbg_m)(k)$ which is only well-defined up to $W_{E_8}$. However, we are only interested in this assignment insofar as to construct a $G(k)$-orbit. Via the isomorphisms of Theorem \ref{thm: thorne1.10} and Theorem \ref{thm: thorne1.11} we resolve the ambiguity introduced by $W_{E_8}$ when assigning $(C,P)$ to a point $\kappa_C \in \Hom(\Lambda, \mbg_m)(k)$ (resp. $\Hom(\Lambda, \mbg_a)(k)$).
		\end{remark}
		
		We prove this theorem by closely following the proof of \cite[Theorem 3.5]{thorne2016arithmetic}.
		
		\begin{proof}
			Let $(C,P) \in \mcT^\sram(k)$, let $V = \Lambda/2\Lambda$, and let $T$ be the maximal subtorus of $H$ on which $\theta$ acts by inversion. Let $\gen{\cdot, \cdot}$ be the pairing defined on $\Lambda$ and let $\gen{\cdot, \cdot}_2\: V \ra \mbf_2$ denote the reduction of $\gen{\cdot, \cdot}$ modulo 2. Let $q\: V \ra \mbf_2$ be the quadratic form defined by $q(v) := \frac{\gen{v',v'}}{2} \pmod 2$, where $v'$ is a lift of $v \in V$ to $\Lambda$. We remark that $q$ is well-defined since the bilinear form on any lattice of type $E$ is even. Recall from Remark \ref{rem: assignment-map} that we may view $\kappa_C$ as an element of $\Hom(\Lambda,\mbg_m)(k)$ which is well-defined up to inversion and the Weyl group of $T$. Additionally, we have that $\Lambda$ is endowed with a Galois action from this assignment.
			
			\bigskip
			
			We now use the curve $C$ to produce the data needed for Lurie's construction. Let $\mathcal{L}$ be twice a theta divisor of $\Pic^0(C)$ and let $\wtilde H_\mathcal{L}$ be the associated Heisenberg group (see \cite[Section 1C]{thorne2016arithmetic}). We have the exact sequence 
				\[
					\xym{
						1 \ar[r] & \mbg_m \ar[r] & \heisenberg{\mathcal{L}} \ar[r] & \Pic^0(C)[2] \ar[r] & 1
					}
				\]
			and an injection $\Pic^0(C)[2] \inj \Lambda^\vee/2\Lambda^\vee$ induced from restriction of divisor classes. This gives us the diagram with exact rows via duality and pullback
				\[
					\xym{
						1 \ar[r] & \mbg_m \ar[r] & \heisenberg{\mathcal{L}} \ar[r] & \Pic^0(C)[2] \ar[r] & 1 \\
						1 \ar[r] & \mbg_m \ar[r] \ar@{=}[u] & \wtilde E \ar[r]^\psi \ar[u] & V \ar[r] \ar[u]_\gamma & 1.
					}
				\]
			Note that the commutator pairing on $\heisenberg{\mathcal{L}}$ descends to the Weil pairing on $\Pic^0(C)[2]$. Since by definition the kernel of $\gamma$ is the radical of $\gen{\cdot, \cdot}_2$, it follows that the commutator pairing on $\wtilde E$ descends to $\gen{\cdot, \cdot}_2$ on $V$.
			
			Define a character of $\wtilde E$ by $\chi_q(\wtilde e) = \wtilde e^2 (-1)^{q(\psi \wtilde e)}$. Note that $\chi_q$ is well-defined since $\wtilde e^2 \in \mbg_m$. Letting $\wtilde V := \ker \chi_q$ gives us the extension
				\[
					\xym{
						1 \ar[r] & \{\pm 1 \} \ar[r] & \wtilde V \ar[r] & V \ar[r] & 1.
					}				
				\]
			We define $W := H^0(\Pic^0(C), \mathcal{L})$. Note that $\heisenberg{\mathcal{L}}$ acts on $W$ by pullback of sections, so we define $\wtilde V$ to act on $W$ via the surjective homomorphism $\wtilde V \ra \heisenberg{\mathcal{L}}$. If $k = k^\sep$, then this is a $16$-dimensional irreducible representation of $\wtilde V(k^\sep)$ sending $-1$ to $-\id_W$. It is clear this action is Galois equivariant.

			\bigskip
	
			We have now constructed a quadruplet $(\Lambda, \wtilde V, W, \rho)$ satisfying the conditions of Data I in Section \ref{sec: Lurie}. Thus, by Theorem \ref{thm: Lurie construction}, we obtain a simple adjoint group $H_0$ of type $E_8$, an involution $\theta_0$ leaving the maximal torus $T_0 \ssq H_0$ stable, and a representation of $\mathfrak{g}_0 = \mathfrak{h}_0^{d\theta=1}$. We have that $\theta_0$ acts on $T_0$ by $t \mapsto t^{-1}$ and that $T_0$ is canonically identified with $\Hom(\Lambda, \mbg_m)$. From now on we view $\kappa_C$ as a point of $T_0$.
			
			Observe that we have constructed a 16-dimensional representation $\rho\: \mathfrak{g}_0 \ra \mathfrak{gl}(W)$ of a Lie algebra of Dynkin type $D_8$ defined over $k$, so $\mathfrak{g}_0$ is split. Since $\mathfrak{g}_0$ is a split Lie subalgebra of the same rank as $H_0$ by Lemma \ref{lem: h^theta is D8}, we have that $H_0$ is split. 
			
			By Proposition \ref{prop: thorne involution} there is an isomorphism $\varphi\: H \ra H_0$, unique up to $H^\theta(k)$-conjugacy, satisfying $\theta_0 \varphi = \varphi \theta$. By Corollary \ref{cor: component counting} this isomorphism is unique up to $G(k)$-conjugacy as well. 
			
			It follows that the orbit $G(k) \cdot \varphi^{-1}(\kappa_C) \in G(k) \bs X(k)$ is well-defined. By Theorem \ref{conj: e8 3.4} and Theorem \ref{thm: thorne1.11} this orbit is stable (regular semi-simple) and the association $\mcT^\sram(k) \ra G(k) \bs X^\rss(k)$ is bijective when $k=k^\sep$. We are now finished with the $E_8$ case. The proof of the $\mathfrak{e}_8$ case is nearly identical, with the maximal tori of $H$ replaced by Cartan subalgebras of $\mathfrak{h}$. 
		\end{proof}

	\subsection{Proof of Theorem \ref{conj: e8 3.6}}

		To prove Theorem \ref{conj: e8 3.6}, one could transplant the proof of \cite[Theorem 3.6]{thorne2016arithmetic} as the arguments almost directly apply to our situation. We have chosen to corral the parts of the argument of \cite[Theorem 3.6]{thorne2016arithmetic} that depend only on Lurie's construction into Lemma \ref{lem: Lurie-and-exts} and Lemma \ref{lem: Lurie-and-descent} in the hopes that it might be of referential convenience. 

		Let $C/k$ be a curve. A \emph{theta group} of $C$ is a central extension of $k$-groups
				\[
				\xym{
					0 \ar[r] & \mbg_m \ar[r] & \Theta \ar[r] & J_C[2] \ar[r] & 0
				}
				\]
		such that the commutator pairing on $\Theta$ descends to the Weil pairing on $J_C[2]$. A morphism of theta groups is a morphism such that the diagram
				\[
				\xym{
					0 \ar[r] & \mbg_m \ar[r] \ar@{=}[d] & \Theta_1 \ar[r] \ar[d]^{\psi} & J_C[2] \ar[r] \ar@{=}[d] & 0 \\
					0 \ar[r] & \mbg_m \ar[r] & \Theta_2 \ar[r] & J_C[2] \ar[r] & 0
				}
				\]
		commutes. Every theta group admits a natural action on the vector space $H^0( \Pic^0(C), 2\Theta)$, with $\Theta$ a theta divisor on $\Pic^0(C)$; an element $(\lambda, \omega)$ acts via pullback of sections by the translation $\tau_\omega$ and scaling by $\lambda$. Every Heisenberg group $\heisenberg{\mathcal{L}_B}$ is a theta group \cite[Section 1C]{thorne2016arithmetic}, and the natural action of $\heisenberg{\mathcal{L}_B}$ on $H^0(\Pic^0(C), \mathcal{L}_B)$ agrees with the action from above. The following well-known result underscores the importance of theta groups.

		\begin{lemma} \label{lem: mum-theta}
			\begin{enumerate}[(a)]
				\item
					Let $C/k$ be a curve and let $\mathcal{L}$ be twice a theta divisor on $\Pic^0(C)$. Then there is a canonical identification between isomorphism classes of theta groups of $C$ and elements of $H^1(k, J_C[2])$ such that the isomorphism class of the Heisenberg group $\heisenberg{\mathcal{L}}$ is identified with the trivial cocycle. 
					
				\item
					Let $A \in J_C(k)$, and let $B \in J_C(k^\sep)$ such that $[2]B = A$. Let $\tau_B\: J_C \ra J_C$ be the translation-by-$B$ morphism and let $\mathcal{L}_B = \tau_B^*\mathcal{L}$. Then the cocycle class $(\sigma \mapsto [B^\sigma - B]) \in H^1(k,J_C[2])$ corresponds to the isomorphism class of $\heisenberg{\mathcal{L}_B}$, and the isomorphism class of $\heisenberg{\mathcal{L}_B}$ is independent of the choice of $B$. 
				\end{enumerate}
		\end{lemma}

		\begin{proof}
			The first statement can be found in \cite[Section 1.6]{cremona2008explicit} for elliptic curves, or in \cite[Section 23]{mumford1970abelian} more generally. The second statement is from \cite[Section 1C]{thorne2016arithmetic}.
		\end{proof}

		The data that a curve $C$ provides to the construction of Lurie is functorial in the theta groups of $C$. We express this fact as the following two lemmas. Note that the data of a marked ramification point on $C$ is not necessary to produce the data needed for Lurie's construction. Rather, it is needed at a later point in the argument to mark an orbit in the appropriate orbit space.

		\begin{lemma} \label{lem: Lurie-and-exts}
			Let $\Lambda$ be an irreducible simply laced root lattice and let $V = \Lambda/2\Lambda$. Let $C$ be any curve such that there exists a surjection $\gamma\: V \ra \Pic^0(C)[2]$ such that the natural pairing on $V$ descends to the Weil pairing. Then the construction of Lurie from Section \ref{sec: Lurie} defines a map
				\[
					\{ \text{theta groups of $J_C[2]$} \} \ra \mathcal{D}(k).
				\]
			Moreover, if $(H_0,\theta_0, T_0, \rho_0)$ is a quadruple in the image of this map, then $H_0$ and $(H_0^{\theta_0})^\circ$ are split and $T_0$ is canonically identified with the torus $\Hom(\Lambda, \mbg_m)$.
				
		\end{lemma}

		\begin{proof}
			The result is established by the first two paragraphs of the proof of \cite[Theorem 3.6]{thorne2016arithmetic}. Alternatively, one can consult the proof of Theorem \ref{conj: e8 3.5} above and replace the choice of Heisenberg group with any theta group of $C$.
		\end{proof}

		\begin{lemma} \label{lem: Lurie-and-descent}
			Let $C/k$ be a curve satisfying the conditions of the previous lemma, let $A \in J_C(k)$ represent a class in $J_C(k)/2$ and let $B \in J_C(k^\sep)$ be such that $[2]B = A$. Let $\psi_B$ be the morphism of theta groups 
				\[
				\xym{
					0 \ar[r] & \mbg_m \ar[r] \ar@{=}[d] & \heisenberg{\mathcal{L}} \ar[r] \ar[d]^{\psi_B} & J_C[2] \ar[r] \ar@{=}[d] & 0 \\
					0 \ar[r] & \mbg_m \ar[r] & \heisenberg{\mathcal{L}_B} \ar[r] & J_C[2] \ar[r] & 0
				}
				\]
			defined over $k^\sep$ induced by translation by $B$. Then the construction of Lurie gives a corresponding morphism of tuples $\psi_B\: (H_0,\theta_0, T_0, \rho_0) \ra (H_B,\theta_B, T_B, \rho_B)$. Moreover, if $i_0, i_B$ are the canonical identifications of $\Hom(\Lambda, \mbg_m)$ with $T_0, T_B$ respectively, then $\psi_B i_0 = i_B$. Furthermore, this morphism is exactly the morphism induced by the image of $\sigma \mapsto [B^\sigma - B]$ under the inclusion $J_C[2] \inj V^\vee$ via \cite[Lemma 2.4]{thorne2016arithmetic}. 
		\end{lemma}
		
		\begin{proof}
			Functoriality is established by the comments preceding in \cite[Lemma 2.4]{thorne2016arithmetic}. The remaining details can be found in paragraph 5 of the proof of \cite[Theorem 3.6]{thorne2016arithmetic}. 
		\end{proof}
		
		We now arrive at the main result of this section

		\begin{reptheorem}{conj: e8 3.6}
			\
			\begin{enumerate}
				\item[$E_8\:$] 
				Fix an $x = (C,P) \in \mcT^\sram(k)$ and by abuse of notation we denote by $x$ the image of $(C,P)$ in $T \sslash W_T$. Let $\pi\: X \ra X \sslash G$ denote the natural quotient map, where $X$ is as in Theorem \ref{conj: e8 3.5}-$E_8$. Note that $X \sslash G$ is canonically isomorphic to $T \sslash W_T$. Let $X_x$ be the fibre of $\pi$ over $x$ and let $J_x = J_C$. Then there is a canonical injective map
				\[
				\frac{J_x(k)}{2} \inj G(k) \bs X_x(k).
				\]
				\item[$\mathfrak{e}_8\:$] 
				Fix an $x = (C,P) \in \mcT^\tram(k)$ and by abuse of notation we denote by $x$ the image of $(C,P)$ in $\mbp \mathfrak{t} \sslash W_T$. Let $\pi\: \mbp X \ra \mbp X \sslash G$ denote the natural quotient map, where $X$ is as in Theorem \ref{conj: e8 3.5}-$\mathfrak{e}_8$. Let $\mbp X_x$ be the fibre of $\pi$ over $x$ and let $J_x = J_C$. Then there is a canonical injective map
				\[
				\frac{J_x(k)}{2} \inj G(k) \bs \mbp X_x(k).
				\]
			\end{enumerate}
		\end{reptheorem}
	
		\begin{proof}
			Let $\Lambda$ be a simply laced lattice of Dynkin type $E_8$ and let $V = \Lambda/2\Lambda$. Let $A \in J_x(k)$ be a rational point and choose $B \in J_x(k^\sep)$ such that $[2]B = A$. As before, we have by Proposition \ref{prop: del Pezzo 1} that there is an isomorphism $V \ra J_x[2]$ of $k$-groups. As before, the natural pairing on $\Lambda$ descends to the Weil pairing on $J_x[2]$.
			
			Translation by $B$ induces an isomorphism of theta groups $\heisenberg{\mathcal{L}} \iso \heisenberg{\mathcal{L}_B}$ defined over $k^\sep$. By Lemma \ref{lem: Lurie-and-exts} we obtain quadruples $(H_0, \theta_0, T_0, \rho_0)$ and $(H_B, \theta_B, T_B, \rho_B)$ from $\heisenberg{\mathcal{L}}$ and $\heisenberg{\mathcal{L}_B}$ respectively, and by Lemma \ref{lem: Lurie-and-descent} we obtain an isomorphism of these tuples $F\: (H_0, \theta_0, T_0, \rho_0) \overset{\sim}{\rightarrow} (H_B, \theta_B, T_B, \rho_B)$ defined over $k^\sep$. From Lemma \ref{lem: Lurie-and-descent}, we see that the cocycle $\sigma \mapsto F^{-1}F^\sigma$ is canonically identified with the cocycle $\sigma \mapsto [B^\sigma - B]$. Note if $A=B=0$ then all of these isomorphisms are in fact identity maps.
						
			By Remark \ref{rem: assignment-map} we choose a point in $\Hom(\Lambda, \mbg_m)^\rss (k)$ lying over $x \in T \sslash W_T(k)$, which we will denote by $\kappa_C$. We use $\kappa_C$ to construct a $G(k)$-orbit in $X_x(k)$. As in Theorem \ref{conj: e8 3.5} we obtain morphisms 
				\begin{align*}
					\varphi_0\: & (H, \theta, T, \rho) \ra (H_0, \theta_0, T_0, \rho_0) \\
					\varphi_B \: & (H, \theta, T, \rho) \ra (H_B, \theta_B, T_B, \rho_B) 
				\end{align*}
			from the standard tuple defined over $k$. The morphisms $\varphi_0, \varphi_B$ are unique up to $G(k)$-conjugacy. Via the canonical identifications of $\Hom(\Lambda, \mbg_m)$ with $T_0, T_B$ we obtain points $\kappa_C^0 \in T_0(k), \kappa_C^B \in T_B(k)$. We have that $\varphi_0^{-1}(\kappa_C^0) \in X_x(k)$ is the point constructed in Theorem \ref{conj: e8 3.5} and we obtain the corresponding $G(k)$-orbit $G(k) \cdot \varphi_0^{-1}(\kappa_C^0)$. Similarly, we obtain the point $\varphi_B^{-1}(\kappa_C^B)$ and the orbit $G(k) \cdot \varphi_B^{-1}(\kappa_C^B)$. Note by Remark \ref{rem: W_T ambiguity} the $G(k)$-orbits $G(k) \cdot \varphi_0^{-1}(\kappa_C^0)$ and $G(k) \cdot \varphi_B^{-1}(\kappa_C^B)$ are independent of the choice of $\kappa_C \in \Hom(\Lambda, \mbg_m)^\rss (k)$ lying over $x$.
						
			Note that $F_x := \varphi_B^{-1} F \varphi_0$ is an automorphism of $H$ which commutes with $\theta$. Additionally, we see that the image under $F$ of $T_0$ is exactly $T_B$ by Lemma \ref{lem: Lurie-and-descent}, and again by Lemma \ref{lem: Lurie-and-descent} have that $F(\kappa_C^0) = \kappa_C^B$. Thus, both orbits $G(k) \cdot \varphi_0^{-1}(\kappa_C^0)$ and $G(k) \cdot \varphi_B^{-1}(\kappa_C^B)$ lie in the slice $X_x$ and we have that $F$ induces an automorphism $F_x \: X_x \ra X_x$ defined over $k^\sep$ sending the orbit of $\varphi_0^{-1}(\kappa_C^0)$ to the orbit of $\varphi_B^{-1}(\kappa_C^B)$.

			As in the proof of \cite[Theorem 3.6]{thorne2016arithmetic} we have a canonical bijection from \cite[Proposition 1]{bhargava2014arithmetic}
				\[
					G(k) \bs X_x(k) \iso \ker( H^1(k, Z_G(\varphi_0^{-1}(\kappa_C^0))) \ra H^1(k,G))
				\]
			under which the orbit $G(k) \cdot \varphi_0^{-1}(\kappa_C^0)$ is sent to the zero element and under which the orbit $G(k) \cdot \varphi_B^{-1}(\kappa_C^B)$ is sent to the cocycle $\sigma \mapsto F_x^{-1}F_x^\sigma$. We also have the canonical isomorphisms
				\[
					Z_G(\varphi_0^{-1}(\kappa_C^0)) \iso Z_{G_0}(\kappa_C^0) \iso \Im(V \ra V^\vee)
				\]
			from \cite[Corollary 2.9]{thorne2013vinberg}. By Proposition \ref{prop: del Pezzo 1} we have that $\Im(V \ra V^\vee) \iso J_x[2]$. Thus there is an injection
				\[
					G(k) \bs X_x(k) \inj H^1(k, Z_G(\varphi_0^{-1}(\kappa_C^0))) \iso H^1(k, J_x[2]).
				\]

			The map in the statement of Theorem \ref{conj: e8 3.6} is defined by sending $A \in J_x(k)/2$ to the orbit $G(k) \cdot \varphi_B^{-1}(\kappa_C^B)$. We have shown that the coboundary map $\delta\: J_x(k)/2 \inj H^1(k, J_x[2])$ factors through $G(k) \bs X_x(k)$, so we see that the mapping $A \mapsto G(k) \cdot \varphi_B^{-1}(\kappa_C^B)$ is injective. This finishes the proof in the $E_8$ case. The $\mathfrak{e}_8$ case is obtained similarly.			
		\end{proof}
	

\section{Acknowledgements}

I am grateful to Nils Bruin for the careful and detailed comments throughout the writing of this article.


\end{document}